\documentclass[11pt, leqno]{amsart}
\allowdisplaybreaks
\usepackage[T1]{fontenc}
\usepackage[utf8]{inputenc}  
\usepackage{amsfonts,amssymb, amscd}
\usepackage{esint}
\usepackage{amsmath}
\usepackage{lmodern}
\usepackage{mathrsfs} 
\usepackage{amsthm}
\usepackage{qsymbols}
\usepackage{mathtools}
\mathtoolsset{showonlyrefs}
\usepackage{dsfont}
\usepackage{graphicx}
\usepackage{latexsym}
\usepackage{chngcntr}
\usepackage[noadjust]{cite}
\usepackage{paralist}
\usepackage{bm}
\usepackage{tikz-cd}
\usepackage{nicefrac}
\usepackage{float}
\usepackage{csquotes}

\usepackage[shortlabels]{enumitem}
\setlist[enumerate,1]{label = \normalfont(\arabic*), ref = (\arabic*)}

\usepackage{xpatch}
\xpatchcmd{\paragraph}{\normalfont}{{\normalfont\bfseries}}{}{}

\makeatletter
\newcommand{\proofstep}[1]{
  \par
  \addvspace{\medskipamount}
  \textit{#1\@addpunct{.}}\enspace\ignorespaces
}
\makeatother

\usepackage{todonotes}
\newtheorem{theorem}{Theorem}[section]
\newtheorem{lemma}[theorem]{Lemma}
\newtheorem{proposition}[theorem]{Proposition}
\newtheorem{corollary}[theorem]{Corollary}
\newtheorem{taggedtheoremx}{Assumption}
\newenvironment{assumption}[1]
{\renewcommand\thetaggedtheoremx{#1}\taggedtheoremx}
{\endtaggedtheoremx}
\theoremstyle{definition}
\newtheorem{definition}[theorem]{Definition}
\theoremstyle{remark}
\newtheorem{remark}[theorem]{Remark}
\newtheorem{example}[theorem]{Example}
\usepackage[plainpages=false,pdfpagelabels,
citecolor=red]{hyperref}
\usepackage{xcolor}
\hypersetup{
	colorlinks,
	linkcolor={cyan!90!black},
	citecolor={magenta},
	urlcolor={green!40!black}
}

\newcommand{\R}{\mathbb{R}}
\newcommand{\C}{\mathbb{C}}
\newcommand{\N}{\mathbb{N}}

\newcommand{\D}{\mathbb{D}}

\renewcommand{\L}{\mathrm{L}}
\let\accentH\H              
\renewcommand{\H}{\mathrm{H}}
\newcommand{\W}{\mathrm{W}}
\newcommand{\Cont}{\mathrm{C}}

\newcommand{\one}{\mathbf{1}}

\renewcommand{\i}{\mathrm{i}}
\newcommand{\e}{\mathrm{e}}

\renewcommand{\Re}{\operatorname{Re}}
\renewcommand{\Im}{\operatorname{Im}}

\newcommand{\les}{\lesssim}

\newcommand{\ccdot}{\, \cdot \,}
\renewcommand{\d}{ \partial }
\newcommand{\summ}{\sum\nolimits}
\newcommand{\ddt}{\frac{\mathrm{d}}{\mathrm{d}t}}
\newcommand{\ddxtwo}{\frac{\mathrm{d}^2}{\mathrm{d}x}}

\DeclareMathOperator{\divergence}{div}
\renewcommand{\div}{\divergence}

\DeclareMathOperator{\loc}{loc}

\DeclareMathOperator{\spann}{span}

\DeclareMathOperator{\curl}{curl}

\DeclareMathOperator{\dif}{d}
\DeclareMathOperator{\tr}{tr}

\newcommand{\norm}[1]{\left\lVert#1\right\rVert}
\newcommand{\abs}[1]{\left\lvert#1\right\rvert}

\newcommand{\abss}[1]{\lvert#1\rvert}
\newcommand{\scal}[1]{\left\langle#1\right\rangle}

\newcommand{\sett}[1]{ \left\{#1\right\} }

\makeatletter
\providecommand{\leftsquigarrow}{
  \mathrel{\mathpalette\reflect@squig\relax}
}
\newcommand{\reflect@squig}[2]{
  \reflectbox{$\m@th#1\rightsquigarrow$}
}
\makeatother

\title[Eigenvalue inequalities for elliptic operators]{The Levine--Weinberger and Friedlander--Filonov inequalities for some classes of elliptic operators}

\author[T. Schmatzler]{Timotheus Schmatzler}
\address{Matematiska institutionen, Stockholms universitet, 106 91 Stockholm, Sweden}
\email{timotheus.schmatzler@math.su.se}

\begin{document}

\begin{abstract}
    We consider the eigenvalue problem for certain classes of elliptic operators, namely inhomogeneous membrane operators $ L = \tfrac{1}{ \rho } ( -\Delta + V ) $ and divergence form operators $ L = -\div A \nabla $, on bounded domains.
    For these operators, we prove ordering inequalities between the Dirichlet and the Neumann eigenvalues, generalizing results of Levine--Weinberger and Friedlander--Filonov for the Laplacian.
    We take inspiration from their proofs and derive sufficient conditions on the coefficients of the operator that ensure that the inequalities remain valid. 
\end{abstract}

\maketitle

\section{Introduction}

Let $ \Omega \subseteq \R^d $, $ d \in \N $, be a bounded Lipschitz domain\footnote{here, a domain is an open, non-empty and connected subset of $ \R^d $}. 
In this article we study the eigenvalues of two classes of elliptic operators, namely inhomogeneous membrane operators, sometimes with a potential, 
\begin{equation}
    \label{eq: def inhom membr op}
    L = \tfrac{1}{ \rho } ( -\Delta + V ) 
\end{equation}
and divergence form operators 
\begin{equation}
    \label{eq: def div form op}
    L = -\div A \nabla  
\end{equation}
on the domain $ \Omega $. 
Here, $ \rho $ is a positive density function, $ V $ is a bounded real potential and $ A $ is a symmetric and uniformly elliptic coefficient matrix,  
see Section~\ref{sec: elementary spec theo} for precise definitions and assumptions.

To avoid unnecessary case distinctions, in the preliminaries we work with elliptic operators of the form
\begin{equation}
    \label{eq: gen form of L}
    L = \tfrac{1}{ \rho } ( -\div A \nabla + V ) \, . 
\end{equation}
We are interested in the eigenvalue problem with a Dirichlet boundary condition,  
\begin{equation}
    \label{eq: dirichlet poblem for L}
    \left\{
    \begin{aligned}
        -\div A \nabla \varphi + V \varphi
        &= \lambda \rho \varphi  
        \, \ & \text{ in } \Omega \, , \\
        \varphi &= 0 \, \ 
        & \text{ on } \d \Omega \, ,
    \end{aligned}
    \right.
\end{equation}
and with a Neumann boundary condition, 
\begin{equation}
    \label{eq: neumann problem for L}
    \left\{
    \begin{aligned}
        -\div A \nabla \psi + V \psi
        &= \mu \rho \psi 
        \, \ & \text{ in } \Omega \, , \\
        \d^A_\nu \psi &= 0 \, \ 
        & \text{ on } \d \Omega \, ,  
    \end{aligned}
    \right.
\end{equation}
where $ \d^A_\nu \psi = A \nabla \psi \cdot \nu $ is the normal trace of $ A \nabla \psi $ at the boundary.
We will prove ordering relations comparing the Dirichlet and the Neumann eigenvalues of $ L $, more precisely, upper bounds for the Neumann eigenvalues in terms of the Dirichlet ones. 

The simplest operator $ L $ of this form is the negative Laplacian $ -\Delta $.
It is a standard result that the respective spectra are discrete and given by sequences of eigenvalues
\begin{align*}
    0 < \lambda_1 < \lambda_2 
    \leq \ldots \leq \lambda_n \to \infty \, , \\
    0 = \mu_1 < \mu_2 \leq \ldots 
    \leq \mu_n \to \infty \,  
\end{align*}
repeated according to multiplicity (where $ \lambda_k $ denote the Dirichlet eigenvalues and $ \mu_k $ the Neumann ones), which are described by the variational principles  
\begin{equation}
    \label{eq: Delta rayleigh}
    \lambda_k = \min_{ \substack{  U \subseteq \H^1_0 ( \Omega ) 
    \\ \dim U = k } } \max_{ u \in U } \,
    \frac{ \int_\Omega \abs{ \nabla u }^2 }{ \norm{u}^2 } \ , 
    \qquad 
    \mu_k = \min_{ \substack{ V \subseteq \H^1 ( \Omega ) 
    \\ \dim V = k } } \max_{ v \in V } \,
    \frac{ \int_\Omega \abs{ \nabla v }^2 }{ \norm{v}^2 } \ .
\end{equation}
Here and throughout, we denote by $ \norm{u} = ( \int_\Omega \abs{u}^2 )^{1/2} $ the $ \L^2 $-norm of a square-integrable function $ u : \Omega \to \C $.
For the Laplace operator, there are many results comparing the Dirichlet and the Neumann eigenvalues.
From~\eqref{eq: Delta rayleigh} there follows the trivial inequality $ \mu_k \leq \lambda_k $ for all $ k \in \N $, but in fact stronger inequalities hold true. 
One iconic result is Pólya's inequality $ \mu_2 < \lambda_1 $ in planar domains, which was later generalized to $ \mu_{k+1} < \lambda_k $ for all $ k \in \N $ and in all dimensions $ d \geq 2 $ by Friedlander (non-strict inequality in $ \Cont^1 $-domains, see~\cite{friedlander_91}) and Filonov (strict inequality in Lipschitz domains, see~\cite{filonov04}). 
For convex domains, Levine and Weinberger improved this to $ \mu_{k+d} \leq \lambda_k $, with strict inequality when $ \d \Omega $ is sufficiently regular (see~\cite{levine_weinberger}). 
In dimension $ d=2 $, it was proved recently that this also holds true in simply connected domains, see~\cite{rohleder_ineq_neumann_dirichlet_2023},
and it is conjectured that $ \mu_{k+d} \leq \lambda_k $ holds for all bounded Lipschitz domains in $ \R^d $, see Conjecture~4.2.42 in~\cite{LMP23}. 

This kind of inequalities can give interesting information on the geometry of eigenfunctions. For example, a Neumann eigenfunction $ \psi $ associated to an eigenvalue no larger than $ \lambda_1 $ cannot have an interior nodal domain; note that this result carries over to more general elliptic operators, provided they satisfy a unique continuation property.  
See also~\cite{CMS19} where further motivation for inequalities of the form $ \mu_k \leq \lambda_1 $ is discussed.  
For large values of $ k $, note also the recent preprint~\cite{freitas_24}, where the inequality $ \mu_{ k+ p(k) } \leq \lambda_k $, $ k $ sufficiently large, is proved for a wide class of domains, where $ p(k) = \lfloor c_d k^{ 1- 1/d } \rfloor $ and $ c_d $ is a dimensional constant. 

In this article we investigate to what extent some of the aforementioned inequalities of the form $ \mu_{k+r} \leq \lambda_k $ carry over to more general elliptic operators $ L = \tfrac{1}{ \rho } ( -\div A \nabla +V ) $. 
We denote their Dirichlet and Neumann eigenvalues by $ \lambda^L_k $ resp.\ $ \mu^L_k $, see Section~\ref{sec: elementary spec theo} for precise definitions. 
In the one-dimensional case, it is not hard to see 
that certain elliptic operators do not satisfy Pólya's inequality $ \mu_2^L \leq \lambda^L_1 $. 
For more concrete examples, in~\cite{cheng_strings_12} it is proven that the one-dimensional vibrating string operator $ L = -\tfrac{1}{ \rho } \ddxtwo $ with a concave density $ \rho $ satisfies the reversed Pólya inequality $ \mu^L_2 \geq \lambda^L_1 $, and in~\cite{rohleder_schroed_ev_21} a similar inequality for one-dimensional Schrödinger operators with a symmetric convex potential is shown.

Inspired by the proofs in~\cite{filonov04} and in~\cite{rohleder_schroed_ev_21}, we derive conditions on the density $ \rho $, the potential $ V $ and the elliptic coefficient matrix $ A $ such that the associated operator $ L $ satisfies ordering relations of the type $ \mu^L_{k+r} \leq \lambda^L_k $ for certain $ r $ and $ k=1 $ or all $ k \in \N $. 
In particular, we prove in Theorem~\ref{thm: convex densities} that the weighted Schrödinger operator $ \tfrac{1}{ \rho } (-\Delta + V) $ with convex density and concave potential on a convex domain satisfies the Pólya type inequality $ \mu^L_d \leq \lambda^L_1 $, and in Theorem~\ref{thm: thm frifi for A} that divergence form operators $ -\div A \nabla $ fulfill the Friedlander--Filonov inequality $ \mu^L_{k+1} \leq \lambda^L_k $, $ k \in \N $, if the matrix $ A(x) $ has an eigenpair independent of $ x \in \Omega $.

This article is structured as follows. 
In Section~\ref{sec: elementary spec theo}, we give precise definitions of the weighted elliptic operators under consideration and list some useful elementary spectral properties. 
In Section~\ref{sec: abstract formulation}, we state an abstract version of the method used in both~\cite{levine_weinberger} and~\cite{filonov04} to prove inequalities of the form $ \mu_{k+r} \leq \lambda_k $, $ k \in \N $, which we later apply to our operators. 
In Section~\ref{sec: lw ineq}, we derive conditions on the density $ \rho $ and the potential $ V $ such that the weighted Schrödinger operator $ L = \tfrac{1}{ \rho } ( -\Delta + V ) $ on a convex domain satisfies inequalities of the Levine--Weinberger type $ \mu^L_{k+r} \leq \lambda^L_k $, $ k \in \N $, or of the Pólya-type $ \mu^L_r \leq \lambda^L_1 $ for some $ r \in \N $. 
In Section~\ref{sec: frifi}, we adapt Filonov's proof to derive the inequalitiy $ \mu^L_{k+1} \leq \lambda^L_k $, $ k \in \N $,  for certain operators of the form $ L = -\div A \nabla $ and $ L = -\tfrac{1}{ \rho } \Delta $ on bounded Lipschitz domains. 
In two dimensions, we interpret the result for inhomogeneous membranes in the framework of complex analysis, and compare it to earlier work by Nehari~\cite{nehari_58} and Bandle~\cite{bandle_72}. 
Finally, the Appendix contains the proof of an integration by parts formula on convex domains used in Section~\ref{sec: lw ineq}.

\section{Elliptic operators and elementary spectral theory}
    \label{sec: elementary spec theo}

In this section, we recall a few elementary definitions and properties of elliptic eigenvalue problems.   
We consider differential expressions $ L = \tfrac{1}{ \rho } ( -\div A \nabla + V ) $, where $ \rho : \Omega \to (0, \infty ) $ and $ V : \Omega \to \R $ are two measurable functions called \emph{density} resp.\ \emph{potential}, and $ A : \Omega \to \R^{ d \times d } $ is a symmetric uniformly elliptic coefficient matrix, see Assumption~\ref{ass: reg of A} below.
Throughout the article we work under the following assumptions on $ \rho $, $ V $ and $ A $. 
\begin{assumption}{($ \rho $)}
    \label{ass: reg of rho}
    The density function $ \rho : \Omega \to (0, \infty ) $ is bounded above and away from zero, i.e., $ \rho $, $ \tfrac{1}{ \rho } \in \L^\infty ( \Omega ) $. 
\end{assumption}
\begin{assumption}{($ V $)}
    \label{ass: reg of V}
        The potential $ V : \Omega \to \R $ is bounded, i.e., $ V \in \L^\infty ( \Omega ) $.
\end{assumption}
\begin{assumption}{($ A $)}
    \label{ass: reg of A}
    The coefficient matrix $ A $ is bounded, symmetric and uniformly elliptic.
    In other words, $ A \in \L^\infty ( \Omega , \R^{ d \times d } ) $ and there exists a constant $ c>0 $ such that for each $ x \in \Omega $, $ A(x) \in \R^{ d \times d } $ is a symmetric matrix with 
\begin{equation}
    A(x) \xi \cdot \xi \geq c \abs{ \xi }^2 \ , 
    \quad \xi \in \R^d \, . 
\end{equation}
\end{assumption}
To treat the eigenvalue problems~\eqref{eq: dirichlet poblem for L} and~\eqref{eq: neumann problem for L}, we work in the weighted Hilbert space 
\begin{equation}
    \label{eq: def of L2 rho}
    \L^2_\rho ( \Omega ) 
    = \{ u \in \L^1_{ \loc } ( \Omega ) : 
    \rho^{1/2} u \in \L^2 ( \Omega ) \} \, 
\end{equation}
with inner product 
\begin{equation}
    \scal{u,v}_\rho = \int_\Omega \rho u \bar{v} \,  
\end{equation}
and denote the norm as $ \norm{u}^2_\rho = \scal{u,u}_\rho $.
Since $ \rho $ is bounded above and away from zero, the norms $ \norm{ \ccdot }_\rho $ and $ \norm{ \ccdot } $ are equivalent, so $ \L^2_\rho ( \Omega ) $ equals the usual Lebesgue space $ \L^2 ( \Omega ) $.
Similarly the weighted Sobolev space 
\begin{equation}
    \H^1_\rho ( \Omega ) 
    = \{ u \in \L^2_\rho ( \Omega ) 
    : \nabla u \in \L^2 ( \Omega )^d \} 
    \subseteq \L^2_\rho ( \Omega ) 
\end{equation}
with norm $ \norm{ \ccdot }^2_{ 1, \rho } = \norm{ \ccdot }^2_\rho + \norm{ \nabla \ccdot }^2 $ is equal to the usual Sobolev space $ \H^1 ( \Omega ) $ with equivalent norms. 
We write $ \H^1_{ \rho , 0 } ( \Omega ) $ for the completion of $ \Cont_c^\infty ( \Omega ) $ under the norm $ \norm{ \ccdot }_{1, \rho } $. Alternatively, $ \H^1_{ \rho , 0 } ( \Omega ) $ is just the space $ \H^1_0 ( \Omega ) $ endowed with the norm $ \norm{ \ccdot }_{1, \rho } $.

Consider the symmetric sesquilinear form 
\begin{equation*}
    \mathfrak{a} (u,v)
    = \mathfrak{a}_\text{N} (u,v) 
     = \int_\Omega A \nabla u \nabla \bar{v} + V u \bar{v} 
     \ , \quad u,v \in \H^1_{ \rho } ( \Omega ) \, ,
\end{equation*}
and denote by $ \mathfrak{a}_\text{D} $ its restriction to $ \H^1_{ \rho , 0 } ( \Omega ) \times \H^1_{ \rho , 0 } ( \Omega ) $. 
Then $ \mathfrak{a}_\text{D} $ and $ \mathfrak{a}_N $ are symmetric, bounded below and closed on the Hilbert space $ \L^2_\rho ( \Omega ) $. 
By standard methods, the associated operators $ L_\text{D} $ resp.\ $ L_\text{N} $ are self-adjoint realizations of the differential expression $ L = \tfrac{1}{ \rho } ( -\div A \nabla + V ) $ on $ \L^2_\rho ( \Omega ) $, with domains 
\begin{align}
    \mathcal{D} ( L_\text{D} ) 
    &= \{ u \in \H^1_\rho ( \Omega ) : 
    \tfrac{1}{ \rho } \div A \nabla u \in \L^2_\rho ( \Omega ) , \ 
    u |_{ \d \Omega } = 0 \} \, , \\
    \mathcal{D} ( L_\text{N} ) 
    &= \{ u \in \H^1_\rho ( \Omega ) : 
    \tfrac{1}{ \rho } \div A \nabla u \in \L^2_\rho ( \Omega ) , \ 
    \d^A_\nu u = 0 \} \, . 
\end{align}
(Note that $ \div A \nabla u \in \L^2 ( \Omega ) $ ensures that $ A \nabla u \cdot \nu = \d^A_\nu u $ is well-defined in $ \H^{-1/2} ( \d \Omega ) $, see~\cite[Chapter~III.2]{galdi_navier_stokes} for details.)
Furthermore $ L_\text{D} $ and $ L_\text{N} $ are bounded below (e.g., by the essential infimum of $ V $) and have compact resolvents. 
Therefore their respective spectra are given by sequences 
\begin{align}
    -\infty < \inf_\Omega V < \lambda^L_1 \leq \lambda^L_2 
    \leq \ldots \leq \lambda^L_n \to \infty \, , \\
    -\infty < \inf_\Omega V 
    \leq \mu^L_1 \leq \mu^L_2 \leq \ldots 
    \leq \mu^L_n \to \infty \, ,
\end{align}
repeated according to multiplicity. 
The associated eigenfunctions $ \{ \varphi^L_k \}_{ k \in \N } \subseteq \H^1_{ \rho , 0 } ( \Omega ) $ of $ L_\text{D} $ resp.\ $ \{ \psi^L_k \}_{ k \in \N } \subseteq \H^1_\rho ( \Omega ) $ of $ L_\text{N} $ can be chosen real-valued and such that they form two orthonormal bases of $ \L^2_\rho ( \Omega ) $. 
Moreover, for $ k \in \N  $, $ l \in \N_0 $, the eigenvalues are described by the minimax of Rayleigh quotients 
\begin{equation}
    \label{eq: variational principle for L}
    \begin{aligned}
        \lambda^L_{k+l} &= \min_{ 
        \substack{ U \subseteq \H^1_{ \rho , 0 } ( \Omega ) \\
        \dim U = k , \ U \perp_\rho \, \varphi^L_1 , \ldots , \varphi^L_l } } 
        \max_{ u \in U } \,
        \frac{ \int_\Omega A \nabla u \cdot \nabla \bar{u} 
        + V \abs{u}^2 }
        { \norm{u}_\rho^2 } \ , \\
        \mu^L_{k+l} &= \min_{ 
        \substack{ U \subseteq \H^1_\rho ( \Omega ) \\
        \dim U = k , \ U \perp_\rho \, \psi^L_1 , \ldots , \psi^L_l } } 
        \max_{ u \in U } \,
        \frac{ \int_\Omega A \nabla u \cdot \nabla \bar{u} 
        + V \abs{u}^2 }
        { \norm{u}_\rho^2 } \ .
    \end{aligned}
\end{equation}
(Here we use the notation $ U \perp_\rho \varphi $ as shorthand for $ \scal{ u , \varphi }_\rho = 0 $ for all $ u \in U $.)
Furthermore, for $ k=1 $ any minimizing subspace is spanned by an eigenfunction corresponding to $ \lambda^L_{l+1} $ resp.\ to $ \mu^L_{l+1} $.

We will need the following result on elliptic boundary regularity for $ L $, which follows from~\cite[Proposition~4.8]{AGMT10} (since convex domains are quasi-convex, see formula~(4.16) in~\cite{AGMT10}; note also that, although the authors consider only Schrödinger operators, the result for $ \tfrac{1}{ \rho } ( -\Delta + V ) $ follows a posteriori since $ \rho $ is bounded above and below). 
For $ \Cont^2 $-domains, see also~\cite[Chapter~6.3]{evans_pdes_98}.
\begin{proposition}
    \label{prop: bdy reg of ef}
    Assume that $ \Omega $ has a $ \Cont^2 $-boundary or is convex, and suppose that $ A \equiv \operatorname{Id} $. Then any Dirichlet eigenfunction $ \varphi \in \H^1_{ \rho , 0 } ( \Omega ) $ of $ L = \tfrac{1}{ \rho } ( -\Delta + V ) $ lies in $ \H^2 ( \Omega ) $. 
\end{proposition}

Under adequate regularity assumptions, second order elliptic operators satisfy a unique continuation property. 
Note that the unique continuation property may fail if the coefficient matrix is only Hölder-regular, see for example~\cite{mandache_counterx_97}, so for simplicity we assume that $ A $ is smooth, say $ A \in \Cont^\infty ( \overline{ \Omega } , \R^{ d \times d } ) $, where for $ l \in \N \cap \{ \infty \} $ we denote 
\begin{equation}
    \label{eq: def of C1 on Omega}
    \Cont^l ( \overline{ \Omega } ) 
    = \Cont^l ( \R^d ) |_{ \overline{ \Omega } } 
    \, . 
\end{equation}
For more precise statements and proofs, we refer to the~\cite{aronszajn} (for $ \Cont^{2,1} $-coefficient matrices), and to~\cite{wolff_ucp_93} for even more general results.

\begin{proposition}
    \label{prop: ef uniq cont}
    Assume $ A $ is smooth on $ \overline{ \Omega } $ and let $ \varphi \in \H^1_{ \rho , 0 } ( \Omega ) $ be a Dirichlet eigenfunction of the operator $ L= \tfrac{1}{ \rho } ( -\div A \nabla + V ) $. 
    If $ \varphi $ vanishes identically on a non-empty open subset of $ \Omega $, then $ \varphi \equiv 0 $ on $ \Omega $. 
\end{proposition}
    
We will use the following consequences of the unique continuation property. Again, these results are well-known for the standard Laplacian, and the proofs carry over without difficulty. 
\begin{corollary}
    \label{cor: no ef dir neum}
    Assume $ A \in \Cont^\infty ( \overline{ \Omega } , \R^{ d \times d } ) $. 
    If $ \varphi \in \H^1_{ \rho , 0 } ( \Omega ) $ is a Dirichlet eigenfunction of $ L = \tfrac{1}{ \rho } ( -\div A \nabla + V ) $ and $ \d^A_\nu \varphi = 0 $ on a relatively open subset $ \omega \subseteq \d \Omega $, then $ \varphi \equiv 0 $ in $ \Omega $. 
    In particular, $ L $ does not have any eigenfunction which satisfies both Dirichlet and Neumann conditions. 
\end{corollary}
\begin{proof}
    The proof is essentially the same as for Lemma~3.1 in~\cite{rohleder_robin_20}, since the density $ \rho $,  the potential $ V $ and the matrix $ A $ can be extended to be admissible on a slightly larger domain $ \Omega' \supseteq \Omega $ with $ \Omega' \cap \d \Omega \subseteq \omega $ ($ A $ admitting a smooth extension by the definition in~\eqref{eq: def of C1 on Omega}).
\end{proof}
\begin{corollary}
    \label{cor: dim trial deriv}
    Assume $ \Omega $ has a $ \Cont^2 $-boundary or is convex. 
    If $ \varphi \in \H^1_{ \rho , 0 } ( \Omega ) $ is a real-valued Dirichlet eigenfunction of the elliptic operator $ L = \tfrac{1}{ \rho } ( -\Delta + V ) $, then the subspace $ \{ b \cdot \nabla \varphi : b \in \C^d \} $ has dimension $ d $ and has trivial intersection with $ \H^1_0 ( \Omega ) = \H^1_{ \rho , 0 } ( \Omega ) $.
\end{corollary}
\begin{proof}
    The proof follows the reasoning from~\cite[Theorem~4.1]{rohleder_lotoreichik}. 
    Observe first that, since $ \Omega $ is bounded, if a function $ u $ is constant in some direction and vanishes at the boundary, then $ u $ must vanish identically in $ \Omega $. This implies that $ \d_1 \varphi , \ldots , \d_d \varphi $ span a $ d $-dimensional space. 
    Indeed, let $ b \in \C^d \setminus \{ 0 \} $, and without loss of generality suppose $ \Re b \neq 0 $. If $ b \cdot \nabla \varphi = 0 $ in $ \Omega $, then, since $ \varphi $ is real, $ \Re ( b \cdot \nabla \varphi ) = \d_{ \Re b } \varphi $ is identically zero, contradicting the observation above. 

    Now suppose $ b \cdot \nabla \varphi $ lies in $ \H^1_0 ( \Omega ) $ for some $ b \in \C^d $. 
    There is an open boundary subset $ \omega \subseteq \d \Omega $ such that $ \nu (x) \cdot \Re b \neq 0 $ for $ x \in \omega $. 
    Since $ \varphi $ vanishes on $ \d \Omega $, the tangential part of $ \nabla \varphi $ vanishes on $ \omega $ (see Lemma~\ref{lem: grad phi normal}). Since $ \Re b $ is not tangential on $ \omega $, together with $ \d_{ \Re b } \varphi = 0 $ in $ \Omega $ it follows that $ \d_\nu u = 0 $ on $ \Omega $.
    But this contradicts Corollary~\ref{cor: no ef dir neum}. 
\end{proof}

\section{An abstract framework for eigenvalue inequalities}
    \label{sec: abstract formulation}

To prove inequalities of the type $ \mu_{k+r} \leq \lambda_k $, the proofs appearing in~\cite{levine_weinberger} and~\cite{filonov04} rely on a very similar argument: finding adequate trial functions to test, along with Dirichlet eigenfunctions, the variational principle for Neumann eigenvalues in~\eqref{eq: variational principle for L}. In this section we formulate this method in an abstract setting for simple reference. 

Let $ H $ be a Hilbert space and $ H_0 , H_1 \subseteq H $ be two subspaces with $ H_0 \subseteq H_1 $. 
Let $ \mathfrak{a} : H_1 \times H_1 \to \C $ be a symmetric and closed sesquilinear form bounded from below, and assume that also the restriction $ \mathfrak{a}_0 = \mathfrak{a} |_{ H_0 \times H_0 } $ is closed. 
Denote by $ L_0 $ resp.\ $ L_1 $ the self-adjoint operators associated to $ \mathfrak{a}_0 $ resp.\ $ \mathfrak{a}_1 $, and suppose that $ L_0 $ and $ L_1 $ have compact resolvents. 
Denote the eigenvalues (counting multiplicity) of $ L_0 $ resp. of $ L_1 $ by 
\begin{align*}
    \lambda_1 \leq \lambda_2 
    \leq \ldots \leq \lambda_n \to \infty \, , \\
    \mu_1 \leq \mu_2 \leq \ldots 
    \leq \mu_n \to \infty \, . 
\end{align*}

\begin{proposition}
    \label{prop: abstract filonov}
    Let $ \lambda \in \R $ and $ k \in \N $ with $ \lambda_k \leq \lambda $, and let $ r \in \N $.
    \begin{enumerate}[(i)]
        \item 
        Suppose there is an $ r $-dimensional subspace $ W \subseteq H_1 $ with $ W \cap H_0 = \{ 0 \} $ such that all $ w \in W $ satisfy 
        \begin{equation}
            \label{eq: def abstract weak sol ev eq}
            \mathfrak{a} (w,u) 
            = \lambda \scal{w,u}_H 
            \quad \text{ for all } u \in H_0 \, 
        \end{equation}
        as well as the estimate 
        \begin{equation}
            \label{eq: form bound for w} 
            \mathfrak{a}(w,w) \leq \lambda \norm{w}^2_H \, . 
        \end{equation}
        Then the eigenvalue inequality $ \mu_{k+r} \leq \lambda $ holds. 
        \item 
        If in addition the inequality~\eqref{eq: form bound for w} is strict for all non-zero $ w \in W $ and if $ L_0 $ and $ L_1 $ do not have a common eigenfunction, then $ \mu_{k+r} < \lambda $. 
        \item 
        Suppose $ r = 1 $ and that $ L_0 $ and $ L_1 $ do not have a common eigenfunction.
        If there exist infinitely many linearly independent elements $ w \in H_1 \setminus H_0 $ satisfying~\eqref{eq: def abstract weak sol ev eq} and~\eqref{eq: form bound for w}, then the inequality $ \mu_{k+1} < \lambda $ is strict.  
    \end{enumerate}
\end{proposition}
To obtain estimates between the Dirichlet and Neumann eigenvalues of elliptic operators of the form $ L= \tfrac{1}{ \rho } ( -\div A \nabla + V ) $, we will set $ H = \L^2_\rho ( \Omega ) $, $ H_0 = \H^1_{ \rho , 0 } ( \Omega ) $ and $ H_1 = \H^1_\rho ( \Omega ) $, and usually choose $ \lambda = \lambda^L_k $. 
Then, one looks for (subspaces of) functions $ w \in \H^1_{ \rho } ( \Omega ) $ solving~\eqref{eq: def abstract weak sol ev eq} and satisfying the estimate~\eqref{eq: form bound for w}. 
In this setting, 
observe that~\eqref{eq: def abstract weak sol ev eq} translates to $ w $ being a solution of the differential equation $ Lw = \lambda w $ in $ \Omega $. 
\begin{proof}[Proof of Proposition~\ref{prop: abstract filonov}]
    Fix $ k \in \N $ and let $ W \subseteq H_1 $ be as in~\emph{(i)}. Consider the subspace
    \begin{equation}
        \label{eq: def of u spanned by ef}
        U = \spann \{ \varphi_1 , \ldots , 
        \varphi_k \} \subseteq H_0 \, , 
    \end{equation} 
    where $ \{ \varphi_l \}_{ l \in \N } $ denotes an orthonormal basis of eigenvectors for $ L_0 $, and set $ E = U+W \subseteq H_1 $.
    Since $ W \cap H_0 = \{ 0 \}$, the dimension of $ E $ is $ k+r $. 
    For any function $ u+w \in E $, where $ u \in U $, $ w \in W $, we have 
    \begin{equation}
        \label{eq: u+w to bound}
        \mathfrak{a} ( u+w , u+w ) 
        = \mathfrak{a} (u,u) + \mathfrak{a} (w,w) 
        + 2 \Re \mathfrak{a} (u,w) \, . 
    \end{equation}
    Writing $ u = \sum a_i \varphi_i $, by orthogonality of the eigenfunctions we have the simple bound 
    \begin{align}
        \mathfrak{a} (u,u) 
        = \sum_{i,j=1}^k a_i \bar{a}_j 
        \mathfrak{a} ( \varphi_i , \varphi_j ) 
        = \sum_{i,j=1}^k a_i \bar{a}_j \lambda_i 
        \scal{ \varphi_i , \varphi_j } 
        \leq \lambda \sum_{i=1}^k 
        \abs{ a_i }^2 \norm{ \varphi_i }^2 
        = \lambda \norm{u}^2_H \, .
    \end{align}
    
    Since $ u \in H_0 $, the assumption~\eqref{eq: def abstract weak sol ev eq} gives that the third term on the right hand side of~\eqref{eq: u+w to bound} is equal to $ 2 \lambda \Re \scal{ u , w }_H $. 
    Finally, by assumption~\eqref{eq: form bound for w} the second term on the right hand side of~\eqref{eq: u+w to bound} is bounded by $ \lambda \norm{w}^2_H $. 
    Summing up, we have 
    \begin{equation}
        \mathfrak{a} ( u+w , u+w ) 
        \leq \lambda \left( \norm{u}^2_H 
        + \norm{ w }^2_H 
        + 2 \Re \scal{ u , w }_H \right)
        = \lambda \norm{ u+w }^2_H \, .
    \end{equation}
    Because the trial subspace $ U \subseteq H_1 $ has dimension $ k+r $, the variational principle for the eigenvalues of $ L_1 $ gives the inequality $ \mu_{k+r} \leq \lambda $. 

    To prove~\emph{(ii)}, suppose for a contradiction that $ \mu_{k+r} = \lambda $. Since $ \dim E = k+r $, we find a non-zero $ z = u+w \in E $ (with $ u \in U $, $ w \in W $) orthogonal to the eigenfunctions $ \psi_1 , \ldots , \psi_{k+r-1} $ of $ L_1 $. 
    The variational principle for the eigenvalues of $ L_1 $ and the fact that $ z \in E $ give  
    \begin{equation}
        \label{eq: all 3 must be equal}
        \mu_{k+r} \norm{z}^2_H 
        \leq \mathfrak{a}_1 (z,z) 
        \leq \lambda \norm{z}^2_H \, . 
    \end{equation}
    Since $ z \neq 0 $ all three terms must be equal. 
    The first equality then implies that $ z $ is an eigenfunction of $ L_1 $, while equality in the second inequality yields $ w = 0 $. 
    But this means that $ z = u \in H_0 $, so $ z $ is also an eigenfunction for $ L_0 $, contradicting the assumption.     

    Finally, let us modify the above argument to prove~\emph{(iii)}. 
    Consider the subspace $ V = \ker ( L_1 - \lambda ) \subseteq H_1 $. 
    Then $ U $ and $ V $ have trivial intersection as $ L_0 $ and $ L_1 $ have no common eigenfunction. 
    Since they are finite-dimensional, by the assumption in~\emph{(iii)} we can find a $ w \in H_1 \setminus H_0 $ satisfying~\eqref{eq: def abstract weak sol ev eq} and~\eqref{eq: form bound for w} such that the sum $ E' = U+V+ \C w \subseteq H_1 $ is direct. 
    Then, for a function $ u+v+ cw \in E' $, $ u \in U $, $ v \in V $, $ c \in \C $, we have 
    \begin{equation}
        \mathfrak{a} ( u+v+cw , u+v+cw ) 
        = \mathfrak{a} (u,u) 
        +\mathfrak{a} (cw,cw) 
        + 2 \Re \mathfrak{a} (u,cw) 
        + \mathfrak{a} (v,v) 
        + 2 \Re \mathfrak{a} ( v , u+cw ) \, . 
    \end{equation}    
    The first three terms we have estimated earlier, and since $ L_1 v = \lambda v $ the last two terms equal $ \lambda ( \norm{v}^2_H + 2 \Re \scal{ v , u+cw }_H ) $. We arrive at 
    \begin{equation}
        \mathfrak{a} ( u+v+cw , u+v+cw ) 
        \leq \lambda \norm{ u+v+cw }^2_H \, . 
    \end{equation}
    Thus there are at least $ \dim E' $-many eigenvalues of $ L_1 $ no larger than $ \lambda $. Therefore, if $ \# F $ denotes the cardinality of a set $ F \subseteq \N $, we arrive at 
    \begin{align}
        \# \{ j : \mu_j < \lambda \} 
        &= \# \{ j : \mu_j \leq \lambda_k \} 
        - \dim \ker ( L_1 - \lambda ) \\
        &\geq \dim E' - \dim \ker ( L_1 - \lambda ) 
        = k+1 \, . 
        \qedhere
    \end{align}
\end{proof}

\section{The Levine--Weinberger inequality for weighted Schrödinger operators on convex domains}
    \label{sec: lw ineq}
On convex domains, Levine and Weinberger proved in~\cite{levine_weinberger} that the inequality $ \mu_{k+d} \leq \lambda_k $ for all $ k \in \N $ holds true for the Laplace operator $ L= -\Delta $. 
The essential idea of their proof is to apply Proposition~\ref{prop: abstract filonov} with the trial subspace $ W $ spanned by partial derivatives $ \d_j \varphi $ of one Dirichlet eigenfunction. 
For the standard Laplacian, it is clear that $ -\Delta \varphi = \lambda \varphi $ entails $ -\Delta \d_j \varphi = \lambda \d_j \varphi $, thus ensuring condition~\eqref{eq: def abstract weak sol ev eq}. 
Moreover, in convex domains the integration by parts Lemma~\ref{lem: ibp convex} below then yields~\eqref{eq: form bound for w}. 
The idea to use derivatives of Dirichlet eigenfunctions as a trial subspace also appears in~\cite{payne55, rohleder_lotoreichik, aldeghi_23,rohleder_schroed_ev_21, lotoreichik_magnetic_24, aldeghi_schroedinger} to cite a few. 

Under adequate assumptions, the ideas in~\cite{levine_weinberger} can be adapted to elliptic operators of the form $ L = \tfrac{1}{ \rho } ( -\Delta + V ) $. 
The arguments in this section are inspired by~\cite{rohleder_schroed_ev_21}, where an analog of the Levine--Weinberger inequality was proven for Schrödinger operators. 

Recall Assumptions~\ref{ass: reg of rho} and~\ref{ass: reg of V}. 
In this section we work for simplicity in a real-valued setting, and assume that the domain $ \Omega $ is piecewise regular according to the following definition. 
\begin{definition}
    \label{def: piecewise smooth domain}
    A Lipschitz domain $ \Omega \subseteq \R^d $ is called piecewise smooth if there is a closed subset $ \Sigma \subseteq \d \Omega $ of boundary measure zero such that $ \d \Omega \setminus \Sigma $ has finitely many connected components, all of which are smooth (i.e., locally given by the graph of a $ \Cont^\infty $-function).  
\end{definition}
The first result of this section generalizes the Pólya inequality $ \mu_2 \leq \lambda_1 $ to operators $ L= \tfrac{1}{ \rho } ( -\Delta +V ) $ under adequate convexity assumptions. 
\begin{theorem}
    \label{thm: convex densities}
    Assume $ \Omega $ is piecewise smooth and convex and that $ \rho, V \in \W^{1, \infty } ( \Omega ) $ (i.e., $ \nabla \rho , \nabla V \in \L^\infty ( \Omega )^d $). 
    \begin{enumerate}[(i)]
        \item
        If the function $ \lambda^L_1 \rho - V $ is convex on $ \Omega $, then the operator $ L = \tfrac{1}{ \rho } ( -\Delta + V ) $ satisfies the eigenvalue inequality $ \mu^L_d \leq \lambda^L_1 $. 
        \item If in addition there is a non-empty open subset of the boundary where all principal curvatures are strictly positive, or if $ D^2 ( \lambda^L_1 \rho - V ) $ is positive definite in an open subset of $ \Omega $, then $ \mu^L_d < \lambda^L_1 $.
        \item
        Under the assumptions in~(i) (resp. in~(ii)), if the domain $ \Omega $, the density $ \rho $ and the potential $ V $ are even with respect to all coordinate axes, then the improved inequality $ \mu^L_{d+1} \leq \lambda^L_1 $ holds (resp. the strict inequality $ \mu^L_{d+1} < \lambda^L_1 $).
        \end{enumerate}
\end{theorem}
The proof will be given later in this section. 
With the method of Proposition~\ref{prop: abstract filonov}, a similar argument gives the following generalization of the Levine-Weinberger inequality on convex domains for operators $ L= \tfrac{1}{ \rho } ( -\Delta +V ) $ with coefficients constant along certain directions. 
\begin{theorem}
    \label{thm: lw when grad has low dim}
    Assume $ \Omega $ is piecewise smooth and convex and suppose that $ \d_b \rho $ and $ \d_b V $ vanish in $ \Omega $ for all $ b $ in a $ r $-dimensional subspace $ E \subseteq \R^d $. 
    Then the eigenvalues of $ L = \tfrac{1}{ \rho } ( -\Delta + V ) $ satisfy $ \mu^L_{k+r} \leq \lambda^L_k $ for all $ k \in \N $. 
    If in addition there is a non-empty open subset of the boundary where all principal curvatures are strictly positive, then $ \mu^L_{k+r} < \lambda^L_k $. 
\end{theorem}
Both proofs hinge on the following integration by parts identity, which is proven for smooth domains and functions in~\cite{levine_weinberger}. 
A proof of this result is given in the Appendix, where we also (essentially) remove the convexity assumption on $ \Omega $. 
Here, we denote by $ D^2 f $ the Hessian matrix of a function $ f $. 
\begin{lemma}
    \label{lem: ibp convex}
    Assume $ \Omega \subseteq \R^d $ is a piecewise smooth convex domain. 
    Let $ \varphi \in \H^2 ( \Omega ) \cap \H^1_0 ( \Omega ) $ be real-valued and $ b \in \R^d $, and set $ v = b \cdot \nabla \varphi \in \H^1 ( \Omega ) $. 
    Then 
    \begin{equation*}
        \int_\Omega \abs{ \nabla v }^2 
        = \int_\Omega \Delta \varphi \, 
        ( b^t D^2 \varphi b ) 
        - \frac{1}{2} \int_{ \d \Omega } \abs{ \nabla \varphi }^2 
        b^t B b \, , 
    \end{equation*}
    where $ B(x) \in \R^{ d \times d } $, $ x \in \d \Omega \setminus \Sigma $, is a symmetric matrix\footnote{see the appendix for a precise definition of $ B $} whose $ d $ eigenvalues are given by the $ (d-1) $ principal curvatures of $ \d \Omega $ at $ x \in \d \Omega $ and their sum. 
\end{lemma}
\begin{remark}
    \label{rem: alt to lw id}
    In~\cite{rohleder_lotoreichik}, this identity is proven for polyhedral domains (i.e., piecewise smooth domains with flat boundaries), where the boundary integral vanishes. Eigenvalue inequalities on convex domains with non-flat boundaries are then obtained through polyhedral domain approximation from outside. 
    While this approach would also work here, we prefer directly applying the formula on piecewise smooth domains: this avoids the need to extend the coefficients of $ L $ beyond $ \Omega $, and the boundary integral can also be used to leverage a strict eigenvalue inequality.
\end{remark}
If $ \Omega $ is convex, $ \d \Omega $ has non-negative curvature so the boundary integral appearing in Lemma~\ref{lem: ibp convex} is non-negative. 
This is the key idea in the proofs of Theorems~\ref{thm: convex densities} and~\ref{thm: lw when grad has low dim}. We compile the core computation in the following lemma.
\begin{lemma}
    \label{lem: eigenf ibp convex}
    Assume $ \Omega $ is a piecewise smooth convex domain and $ \varphi \in \H^1_{ \rho , 0 } ( \Omega ) $ is a (real-valued) Dirichlet eigenfunction of $ L = \tfrac{1}{ \rho } ( -\Delta + V ) $ with eigenvalue $ \lambda $. 
    Let $ b \in \R^d $ and assume that $ \rho $ and $ V $ have a bounded weak derivative in direction $ b $, i.e., $ \d_b \rho $, $ \d_b V \in \L^\infty ( \Omega ) $. 
    Then the function $ v = b \cdot \nabla \varphi $ lies in $ \H^1_\rho ( \Omega ) $ and satisfies the inequality 
    \begin{equation}
        \label{eq: middle of long align}
        \int_\Omega \abs{ \nabla v }^2 
        + V \abs{v}^2 
        \leq \lambda \int_\Omega \rho \abs{v}^2 
        + \int_\Omega \varphi v \, 
         \d_b ( \lambda \rho - V ) \, . 
    \end{equation}
    If $ \int_\Omega b^t B b > 0 $, where $ B $ is as in Lemma~\ref{lem: ibp convex}, then the inequality is strict. 
\end{lemma}
\begin{proof}
    The function $ v $ lies in $ \H^1 ( \Omega ) $ by Proposition~\ref{prop: bdy reg of ef}.  
    Applying Lemma~\ref{lem: ibp convex} and since the boundary integral is non-negative, the eigenvalue equation for $ \varphi $ gives 
    \begin{align}
        \int_\Omega \abs{ \nabla v }^2 + V \abs{v}^2 
        &= \int_\Omega \Delta \varphi \, 
        ( b^t D^2 \varphi b ) 
        - \frac{1}{2} \int_{ \d \Omega } \abs{ \nabla \varphi }^2 
        b^t B b 
        + \int_\Omega V \abs{ \d_b \varphi }^2 \\
        &\leq \int_\Omega \Delta \varphi ( b^t D^2 \varphi b ) 
        + ( \d_b ( V \varphi ) - \varphi \d_b V ) \d_b \varphi \\
        &= \int_\Omega ( - \lambda \rho \varphi + V \varphi ) 
        ( b^t D^2 \varphi b  ) 
        - V \varphi ( b^t D^2 \varphi b ) 
        - \varphi \d_b \varphi \d_b V \\ 
        &= \int_\Omega -\lambda \rho \varphi ( b^t D^2 \varphi b ) 
        - \varphi \d_b \varphi \d_b V \\
        &= \int_\Omega \lambda \rho \, \d_b \varphi \, \d_b \varphi 
        + \varphi \, \d_b \varphi \, \d_b ( \lambda \rho - V ) \, . 
    \end{align}
    For the strict inequality, note that $ \int_\Omega b^t B b > 0 $ and $ b^t B b \geq 0 $ on $ \d \Omega $ give $ \int_\Omega \abs{ \nabla \varphi }^2 b^t B b > 0 $ by Corollary~\ref{cor: no ef dir neum}, so in the second line above the inequality is strict. 
\end{proof}
We now prove Theorem~\ref{thm: lw when grad has low dim}, and Theorem~\ref{thm: convex densities} below.
\begin{proof}[Proof of Theorem~\ref{thm: lw when grad has low dim}]
    In the setting of Proposition~\ref{prop: abstract filonov}, we set $ H = \L^2_\rho ( \Omega ) $, $ H_1 = \H^1_\rho ( \Omega ) $, $ H_0 = \H^1_{ \rho , 0 } ( \Omega ) $, $ L_1 = L_\text{N}  $ and $ L_0 = L_\text{D} $, as well as $ \lambda = \lambda^L_k $. 
    Let $ \{ \varphi_l \}_{ l \in \N } $ denotes as usual an orthonormal basis of real-valued Dirichlet eigenfunctions of $ L $ and consider the subspace 
    \begin{equation}
        \label{eq: def of W lw low dim}
        W = \{ b \cdot \nabla \varphi_k : b \in E \} 
        \subseteq \L^2_\rho ( \Omega ) \, . 
    \end{equation}
    Since $ \Omega $ is convex, by Proposition~\ref{prop: bdy reg of ef} $ W $ is a subspace of $ \H^1_\rho ( \Omega ) $.
    By Corollary~\ref{cor: dim trial deriv}, $ W $ has dimension $ r $ and trivial intersection with $ \H^1_{ \rho , 0 } ( \Omega ) $.
    Let $ b \in E $ and $ w = b \cdot \nabla \varphi \in W $.
    Since $ b \cdot \nabla ( \lambda^L_k \rho - V ) $ vanishes on $ \Omega $, Lemma~\ref{lem: eigenf ibp convex} ensures that $ w $ satisfies the estimate~\eqref{eq: form bound for w}.
    Moreover, we compute 
    \begin{equation}
        -\Delta w + V w 
        = -\d_b \Delta \varphi_k + V w 
        = \d_b ( \lambda^L_k \rho \varphi_k 
        - V \varphi_k ) 
        + V v 
        = \lambda^L_k \rho w 
        + \varphi_k \d_b ( \lambda^L_k \rho 
        - V ) \, . 
    \end{equation}
    As earlier, the last term vanishes, so $ w $ satisfies~\eqref{eq: def abstract weak sol ev eq}: for all $ u \in \H^1_{ \rho , 0 } ( \Omega ) $, 
    \begin{align}
        \int_\Omega \nabla w \nabla u + V w u
        = \int_\Omega ( -\Delta + V ) w \, u 
        = \lambda^L_k \int_\Omega \rho w u \, . 
    \end{align}
    Hence the assertion follows from Proposition~\ref{prop: abstract filonov}\emph{(i)}. 

    Now assume the strict curvature condition. Then there holds strict inequality in Lemma~\ref{lem: eigenf ibp convex}, so we obtain $ \int_\Omega \abs{ \nabla w }^2 + V \abs{w}^2 < \lambda^L_k \int_\Omega \rho \abs{w}^2 $ whenever $ b \neq 0 $. The strict inequality $ \mu^L_{k+r} < \lambda^L_k $ then follows from of Proposition~\ref{prop: abstract filonov}\emph{(ii)} (and Corollary~\ref{cor: no ef dir neum}). 
\end{proof}
\begin{proof}[Proof of Theorem~\ref{thm: convex densities}]
    Consider the subspace 
    \begin{equation}
        \label{eq: def of U proof of conv dens}
        U = \{ b \cdot \nabla \varphi_1 : b \in \R^d \} 
        \subseteq \L^2_\rho ( \Omega ) \, , 
    \end{equation}
    where $ \varphi_1 \in \H^1_{ \rho , 0 } ( \Omega ) $ is a real-valued first Dirichlet eigenfunction of $ L $. 
    Then $ U $ is a subspace of $ \H^1_\rho ( \Omega ) $ by Lemma~\ref{prop: bdy reg of ef} and has dimension $ d $ by Corollary~\ref{cor: dim trial deriv}. 
    For $ v = b \cdot \nabla \varphi_1 \in U $,~\eqref{eq: middle of long align} and an integration by parts give 
    \begin{align}
        \int_\Omega \abs{ \nabla v }^2 + V \abs{v}^2 
        &\leq \lambda^L_1 \int_\Omega \rho \abs{v}^2 
        + \int_\Omega \varphi_1 ( b \cdot \nabla ( \lambda^L_1 \rho 
        - V ) ) \, b \cdot \nabla \varphi_1 \\
        \label{eq: convex computation}
        &= \lambda^L_1 \int_\Omega \rho \abs{v}^2 
        + \frac{1}{2} \int_\Omega ( b \cdot \nabla ( \lambda^L_1 \rho 
        - V ) ) \, b \cdot \nabla ( \varphi_1^2 )  \\ 
        &= \lambda^L_1 \int_\Omega \rho \abs{v}^2  
        - \frac{1}{2} \int_\Omega ( b^t D^2 ( \lambda_1^L \rho 
        - V ) b ) \ \varphi_1^2 \, . 
    \end{align}
    Here, since $ \lambda^L_1 $ is convex, $ D^2 ( \lambda_1^L \rho - V ) $ is a positive semi-definite matrix of Radon measures on $ \Omega $, for details see~\cite[Chapter~6.3]{fine_ppties_of_fct}. 
    But then $ ( b^t D^2 ( \lambda_1^L \rho - V ) b ) \geq 0 $ on $ \Omega $,  
    so the last integral is non-negative, hence the expression in the last line above is bounded by $ \lambda^L_1 \norm{v}^2_\rho $. Thus $ \dim U = d $ gives $ \mu^L_d \leq \lambda^L_1 $. 
    For~\emph{(ii)}, strict eigenvalue inequality follows from a strict inequality in the first line above or by the fact that for any $ b \neq 0 $, it holds $ b^t D^2 ( \lambda^L_1 \rho - V ) b > 0 $ on an open subset of $ \Omega $. 

    Under the additional symmetry assumption in~\emph{(iii)}, the Dirichlet eigenfunction $ \varphi_1 $ is even with respect to all coordinate axes. This is because $ \lambda^L_1 $ is simple by Courant's nodal theorem, and reflecting $ \varphi_1 $ still gives an eigenfunction. 
    Thus, each partial derivative $ \d_j \varphi_1 $, $ j=1, \ldots , d $ is odd in the direction $ j $ and even in all other directions. 
    Similarly also the first Neumann eigenfunction $ \psi_1 $ is even (as well as $ \rho $), therefore we have 
    \begin{equation}
        \int_\Omega \rho \, \d_j \varphi_1 \, \psi_1 
        = 0 \, , 
    \end{equation}
    thus $ U \perp_\rho \psi_1 $. 
    By the variational principle~\eqref{eq: variational principle for L}
    it follows $ \mu^L_{d+1} \leq \lambda^L_1 $. 
\end{proof}
\begin{remark}
    If $ \lambda^L_1 \rho - V $ is convex only along $ r $-many directions for some $ r \in \{ 1 , \ldots , d \} $, that is, $ D^2 ( \lambda^L_1 \rho - V ) (x) $ is positive semi-definite on an $ r $-dimensional subspace $ E \subseteq \R^d $ independent of $ x $, an obvious modification of the above proof gives $ \mu^L_r \leq \lambda^L_1 $.
\end{remark}
\begin{example}
    On a convex and bounded domain $ \Omega \subseteq \R^d $, for any $ \xi \in \R^d $ and strictly positive $ f \in \L^\infty ( \R ) $, the inhomogeneous membrane operator $ L = -\tfrac{1}{ \rho } \Delta $ with density $ \rho (x) = f( \xi \cdot x ) $ satisfies $ \mu^L_{k+d-1} \leq \lambda^L_k $ for all $ k \in \N $ according to Theorem~\ref{thm: lw when grad has low dim}. 
    If we add to the operator $ L $ a potential $ V $ of the form $ V(x) = g( \eta \cdot x ) $ for some function $ g \in \L^\infty ( \R ) $ and $ \eta \in \R^d $, then we have the eigenvalue inequality $ \mu^L_{k+d-2} \leq \lambda^L_k $ for all $ k $, and even $ \mu^L_{k+d-1} \leq \lambda^L_k $ if $ \xi $ and $ \eta $ are colinear.  
\end{example}
\begin{example}
    \label{ex: density x squared convex}
    For any dimension $ d \in \N $, any $ \alpha \geq 1 $ and any convex domain $ \Omega \subseteq \R^d $ the density $ \rho (x) = \abs{x}^\alpha $ is convex. 
    If $ \Omega $ has positive distance to the origin, then $ \rho $ is admissible. 
    By Theorem~\ref{thm: convex densities}, the inhomogeneous membrane operator $ L = - \abs{x}^{ -\alpha } \Delta $ satisfies $ \mu^L_d < \lambda^L_1 $. 
    If $ \Omega \subseteq \R^d $ is convex and symmetric with respect to each coordinate axis, then for $ c>0 $ the operator $ L_c = -\tfrac{1}{ \rho_c } \Delta $ with density $ \rho_c (x) = c + \abs{x}^\alpha $ satisfies $ \mu^{ L_c }_{d+1} < \lambda^{ L_c }_1 $. 
\end{example}

\section{The Friedlander--Filonov inequality for inhomogeneous membranes and divergence form operators}
    \label{sec: frifi}
In this section we aim to generalize the Friedlander--Filonov inequality $ \mu_{k+1} < \lambda_k $, valid for the Laplace operator, to more general elliptic operators. 
In a similar direction, the article~\cite{mazzeo91} gives certain conditions on the geometry of open manifolds for the Laplacian on any compact subdomain to satisfy the inequality $ \mu_{k+1} \leq \lambda_k $, $ k \in \N $, and provides examples where the inequality does not hold. 
Let us also mention the article~\cite{hansson08}, which investigates the question for the three-dimensional Heisenberg Laplacian, and the recent preprint~\cite{FHL24} for Carnot groups. 

Both Friedlander's and Filonov's proofs for the standard Laplacian $ -\Delta $ rely on the fact that complex exponentials $ w(x) = \exp ( \i \xi x ) $ with $ \xi \in \R^d $ solve 
\begin{align}
    -\Delta w &= \abs{ \xi }^2 w \, , \\
    \abs{ \nabla w }^2 
    &= \abs{ \xi }^2 \abs{w}^2 \, , 
\end{align}
hence they satisfy the conditions~\eqref{eq: def abstract weak sol ev eq} and~\eqref{eq: form bound for w} from Proposition~\ref{prop: abstract filonov}. 

For non-constant coefficients, we take inspiration from the Laplacian case and make the ansatz $ w = \exp ( i h(x) ) $ with some function $ h : \Omega \to \C $ to derive conditions on the operator $ L $ such that the requirements of Proposition~\ref{prop: abstract filonov} are satisfied.

\subsection{Inhomogeneous membrane operators}
    \label{subsec: harmgrad densities}
Recall Assumption~\ref{ass: reg of rho} and that $ \Omega \subseteq \R^d $ is a bounded Lipschitz domain. 
We first prove that, under a \emph{harmonic gradient} condition on $ \rho $, the inhomogeneous membrane operator $ L = -\tfrac{1}{ \rho } \Delta $ satisfies the Friedlander--Filonov inequality. 
Later we compare this result to an observation from complex analysis.  
\begin{theorem}
    \label{thm: harm grad density}
    Consider the inhomogeneous membrane operator $ L= -\tfrac{1}{ \rho } \Delta $ on $ \Omega $ and assume that there is a harmonic function $ h : \Omega \to \R $ such that the density $ \rho $ is given by $ \rho = \abs{ \nabla h }^2 $. 
    Then the inequality $ \mu^L_{k+1} \leq \lambda^L_k $ holds for all $ k \in \N $. 
\end{theorem}
\begin{proof}
    In the notation of Proposition~\ref{prop: abstract filonov}, we set again $ H = \L^2_\rho ( \Omega ) $, $ H_1 = \H^1_\rho ( \Omega ) $, $ H_0 = \H^1_{ \rho , 0 } ( \Omega ) $, $ L_1 = L_\text{N}  $ and $ L_0 = L_\text{D} $. 
    For $ \mu >0 $ the function $ w = \exp ( \i \sqrt{ \mu } h ) $ lies in $ \H^1_\rho ( \Omega ) \setminus \H^1_{ \rho, 0 } ( \Omega ) $ and satisfies the differential equation 
    \begin{equation}
        -\Delta w = ( -\i \sqrt{ \mu } \Delta h 
        + \mu \abs{ \nabla h }^2 ) \, \e^{ \i \sqrt{ \mu } h } 
        = \mu \rho w \, , 
    \end{equation}
    as well as the equality  
    \begin{equation}
        \int_\Omega \abs{ \nabla w }^2 
        = \int_\Omega \abs{ \i \sqrt{ \mu } \nabla h }^2 
        \abss{ \e^{ \i \sqrt{ \mu } h } }^2 
        = \mu \int_\Omega \rho \abs{w}^2 \, .
    \end{equation}
    Therefore Proposition~\ref{prop: abstract filonov}\emph{(i)} gives the assertion.
\end{proof}
The condition from Theorem~\ref{thm: harm grad density} is rather restrictive: $ \rho = \abs{ \nabla h }^2 $ for some harmonic $ h $ implies that $ \rho $ is smooth and subharmonic, i.e., $ -\Delta \rho \leq 0 $. But in general the condition seems difficult to verify. In the case of planar domains, we give an alternative equivalent condition in Proposition~\ref{prop: log harm is harm grad}. 
\begin{example}
    \label{ex: rho is green}
    Green's function $ G(x) = c_d \abs{x}^{2-d} $ (for $ d \geq 3 $) or $ G(x) = c_2 \log \abs{x} $ (for $ d=2 $) is harmonic in $ \R^d \setminus \{ 0 \} $, so the inhomogeneous membrane operator $ L = -\tfrac{1}{ \rho } \Delta $ with density $ \rho (x) = \abs{x}^{2-2d} = \tilde{c}_d \abs{ \nabla G (x) }^2 $ satisfies the inequality $ \mu^L_{k+1} \leq \lambda^L_k $, $ k \in \N $, on any domain $ \Omega \subseteq \R^d $ away from the origin. 
\end{example}
\begin{example}
    \label{ex: rho harm grad}
    If the dimension $ d $ is even, then the density $ \rho (x) = \abs{x}^2 $ is given by $ \abs{ \nabla h }^2 $ where $ h(x) = ( x_1^2 - x_2^2 + \ldots - x_d^2 )/ 2 $ (or $ h(x) = x_1 x_2 + x_3 x_4 + \ldots $) is harmonic, hence $ L = -\tfrac{1}{ \abs{x}^2 } \Delta $ satisfies the Friedlander--Filonov inequality on any domain away from the origin. 
    In odd space dimensions $ d $ however, there is no harmonic $ h : \R^d \to \R $ with $ \abs{ \nabla h }^2 = \abs{x}^2 $: The growth condition on $ \nabla h $ would imply that $ h $ is a quadratic polynomial, so $ h(x) = x^t A x + b^t x + c $ for some $ c \in \R $, $ b \in \R^d $ and symmetric $ A \in \R^{ d \times d } $, and one easily shows $ b = 0 $, $ \tr A = 0 $ and $ A^2 = \operatorname{Id} $, which is not possible if $ d $ is odd. 
    
\end{example}
In both examples above, the densities are also convex, so the results can be compared with the ones in Section~\ref{sec: lw ineq}. Note however that here, the domain $ \Omega $ does not have to be convex. \\
\paragraph{The special case of planar domains} 
Let us give some more context for Theorem~\ref{thm: harm grad density} in the two-dimensional case.
Planar domains can be seen as open connected subsets of $ \C $ in a natural way. 
The Riemann mapping theorem asserts that a bounded simply connected domain $ \Omega \subseteq \R^2 $ can be mapped conformally to any other.  
Let $ \phi : \Omega' \to \Omega $ be a conformal transformation, that is, a holomorphic bijection with holomorphic inverse. 
A simple computation shows that, for $ u : \Omega \to \R $ and $ v = u \circ \phi : \Omega' \to \R $, there holds 
\begin{equation}
    \label{eq: delta under conformal}
    \Delta v = \abs{ \phi' }^2 ( \Delta u ) \circ \phi \, . 
\end{equation}
This means that the map $ \phi $ transforms the Dirichlet eigenvalue problem for the standard Laplacian on the domain $ \Omega $ to the eigenvalue problem 
\begin{equation}
    \left\{
    \begin{aligned}
        -\Delta v &= \lambda \abs{ \phi' }^2 v 
        \, \ & \text{ in } \Omega' \, , \\
        v &= 0 \, \ 
        & \text{ on } \d \Omega' \, . 
    \end{aligned}
    \right.
\end{equation}
Assuming that both $ \Omega $ and $ \Omega' $ have regular boundaries\footnote{more precisely, we suppose that $ \phi $ is holomorphic up to the boundary, see~\cite{pommerenke92} for details on the boundary behaviour of conformal maps} 
and because conformal maps preserve angles, $ \phi $ also transforms the Neumann eigenvalue problem for the standard Laplacian on $ \Omega $ into the eigenvalue problem 
\begin{equation}
    \left\{
    \begin{aligned}
        -\Delta v &= \mu \abs{ \phi' }^2 v 
        \, \ & \text{ in } \Omega' \, , \\
        \d_\nu v &= 0 \, \ 
        & \text{ on } \d \Omega' \, . 
    \end{aligned}
    \right.
\end{equation}
In other words, the Dirichlet and the Neumann realizations of the Laplacian $ -\Delta^\Omega $ on $ \Omega $ have the same spectrum as the inhomogeneous membrane operator $ L = -\tfrac{1}{ \rho } \Delta^{ \Omega' } $ with density $ \rho = \abs{ \phi' }^2 $ on $ \Omega' $ with Dirichlet resp.\ Neumann boundary conditions. 
In particular, the strict Friedlander--Filonov inequality $ \mu^L_{k+1} < \lambda^L_k $ for all $ k $ follows directly from the corresponding inequality for the usual Laplacian $ -\Delta^\Omega $. 

Let us relate this to Theorem~\ref{thm: harm grad density}. 
If $ \phi : \Omega' \to \Omega $ is conformal, then $ h = \Re \phi : \Omega \to \R $ is harmonic and $ \abs{ \phi' }^2 = \abs{ \nabla h }^2 $, so $ \rho = \abs{ \phi' }^2 $ fulfills the harmonic gradient condition. 
Theorem~\ref{thm: harm grad density} can thus be seen as a generalization to dimensions $ d \geq 3 $ of this observation.

But even in two dimensions, Theorem~\ref{thm: harm grad density} is more general than the observation via the Riemann mapping theorem. 
Indeed, if $ \phi : D \to \C $ is holomorphic with non-vanishing derivative (but not necessarily injective), then on any Lipschitz domain $ \overline{ \Omega } \subseteq D $, the weight $ \rho = \abs{ \phi' }^2 $ satisfies Assumption~\ref{ass: reg of rho} and the harmonic gradient condition, and so the operator $ L = -\tfrac{1}{ \rho } \Delta^\Omega $ satisfies the Friedlander--Filonov inequality.
\begin{example}
    Let $ \Omega \subseteq \R^2 $ be a bounded simply connected Lipschitz domain with positive distance to the origin. 
    Then the weight $ \rho ( x_1 , x_2 ) = \exp ( \tfrac{ x_1 }{ x_1^2 + x_2^2 } ) $ satisfies the harmonic gradient condition on $ \Omega $.
    Indeed, identifying $ z = x_1 + \i x_2 $, we note $ \rho ( x_1, x_2 ) = \abss{ \e^{ 1/(2z) } }^2 $, and $ z \to \e^{ 1/(2z) } $ admits a primitive function on any simply connected domain.
    Thus, by Theorem~\ref{thm: harm grad density}, the operator $ L = -\tfrac{1}{ \rho } \Delta $ satisfies the inequality $ \mu^L_{k+1} \leq \lambda^L_k $ for all $ k \in \N $. 
    Note that there are simply connected domains $ \Omega \subseteq \C $ where the holomorphic map $ z \mapsto \e^{ 1/ (2z) } $ is not injective.
\end{example}
We note that, on simply connected planar domains, complex analysis translates the harmonic gradient condition into log-harmonicity, which is much simpler to verify. 
\begin{proposition}
    \label{prop: log harm is harm grad} 
    Let $ \Omega \subseteq \R^2 $ be a simply connected domain and $ \rho : \Omega \to (0, \infty ) $ a function. There exists a harmonic function $ h : \Omega \to \R $ with $ \rho = \abs{ \nabla h }^2 $ if and only if 
    \begin{equation}
        \Delta \log \rho = 0 \quad \text{ in } \Omega \, . 
    \end{equation}
    In this case, the family $ \{ \exp ( \i h ) : h \text{ is harmonic with } \rho = \abs{ \nabla h }^2 \} $ spans an infinite-dimensional vector space. 
\end{proposition}
\begin{proof}
    Suppose $ \rho = \abs { \nabla h }^2 $ with $ h $ harmonic. 
    Since $ \Omega $ is simply connected, $ h $ admits a harmonic conjugate, so there is a holomorphic function $ \phi : \Omega \to \C $ with $ h = \Re \phi $. 
    Because $ \abs{ \phi' }^2 = \abs{ \nabla h }^2 = \rho $ does not vanish on $ \Omega $, there is a holomorphic $ \Psi $ such that $ \phi' = \exp \Psi $ (by simple connectedness again). So there holds 
    \begin{equation}
        \rho = \abs{ \nabla h }^2 
        = \abs{ \phi' }^2 
        = \abs{ \exp \Psi }^2 
        = \e^{ 2 \Re \Psi } \, .
    \end{equation}
    Since $ \Re \Psi $ is harmonic, $ \log \rho $ is harmonic. 

    Conversely, suppose that $ \rho = \e^g $ with a harmonic $ g : \Omega \to \R $. As before, we have $ g = \Re \phi $ for some holomorphic $ \phi $ on $ \Omega $. 
    By simple connectedness, the holomorphic function $ \e^{ \phi /2 } $ has a primitive $ \Psi $, and we compute 
    \begin{equation}
        \rho = \e^g = \e^{ \Re \phi }
        = \abss{ \e^{ \phi /2} }^2 
        = \abss{ \Psi' }^2 
        = \abss{ \nabla \Re \Psi }^2 \, , 
    \end{equation}
    where $ h = \Re \Psi $ is harmonic on $ \Omega $. 

    The last assertion is a consequence of the fact that, if $ h = \Re \phi $ is harmonic with $ \rho = \abs{ \nabla h }^2 $, then we get many other harmonic functions via $ h_\theta = \Re ( \e^{ \i \theta } \phi ) $ for $ \theta \in [0, 2 \pi ) $, all of which fulfill $ \rho = \abs{ \nabla h }^2 $. 
    It is not difficult to verify that most of them are linearly independent: using induction over the size of the linear combination, one can show that the family 
    \begin{equation}
        \label{eq: lin indep family}
        \sett{ \e^{ \alpha \Re \phi + \beta \Im \phi } 
        : \alpha , \beta \in \C , \, \alpha^2 + \beta^2 \neq 0 } 
    \end{equation} 
    is linearly independent.
\end{proof}
\begin{corollary}
    \label{cor: strict frifi}
    Assume $ \Omega \subseteq \R^2 $ is a simply connected planar domain and that the density $ \rho $ satisfies $ \Delta \log \rho = 0 $. Then the inhomogeneous membrane operator $ L = -\tfrac{1}{ \rho } \Delta $ on $ \Omega $ satisfies the strict inequality $ \mu^L_{k+1} < \lambda^L_k $ for all $ k \in \N $.
\end{corollary}
\begin{proof}
    This follows from the last assertion of Proposition~\ref{prop: log harm is harm grad} and Proposition~\ref{prop: abstract filonov}~(iii), in the same way as Theorem~\ref{thm: harm grad density}. 
\end{proof}
To illustrate how much simpler complex analysis makes things, we note that $ \Delta \log \rho = 0 $ implies $ \Delta \log \rho^\alpha = 0 $ for any $ \alpha \in \R $, and that the product of two log-harmonic functions is again log-harmonic. 
A priori this cannot be read off from the harmonic gradient condition from Theorem~\ref{thm: harm grad density}.

\begin{remark}
    \label{rem: nehari bandle}
    Corollary~\ref{cor: strict frifi} can be related to two results of Nehari \cite{nehari_58} resp.\ Bandle \cite{bandle_72}, who generalized the Faber-Krahn resp.\ the Szeg\accentH{o}-Weinberger inequality as follows. 
    Assume that $ \Omega \subseteq \R^2 $ is a simply connected domain and $ \rho : \Omega \to (0, \infty ) $ satisfies Assumption~\ref{ass: reg of rho} and is log-subharmonic, i.e., 
    \begin{equation}
        \label{eq: def rho log subharm} 
        -\Delta \log \rho \leq 0 \, 
        \quad \text{ in } \Omega \, . 
    \end{equation} 
    Consider the inhomogeneous membrane operator $ L = -\tfrac{1}{ \rho } \Delta $ on $ \Omega $. 
    Then $ \lambda^L_1 $ is at least as large as the first Dirichlet eigenvalue of the Laplacian on a disk $ \D $ of volume $ \int_\Omega \rho $ (see~\cite{nehari_58}). 
    Moreover, $ \mu^L_1 $ is no larger than the second Neumann eigenvalue of the Laplacian on that same disk (see~\cite{bandle_72}). In particular, there holds the Pólya-type inequality 
    \begin{equation}
        \mu^L_2 \leq \mu^{ -\Delta^\D }_2 
        < \lambda^{ -\Delta^\D }_1 
        \leq \lambda^L_1 \, . 
    \end{equation}
    Both their proofs rely on methods from complex analysis, hence they do not immediately generalize to higher dimensions. 
\end{remark}

\subsection{Divergence form operators}
    \label{subsec: frifi for div form}
Finally, we consider divergence form operators $ L = -\div A \nabla $ on a bounded Lipschitz domain $ \Omega \subseteq \R^d $, where $ A $ is a coefficient matrix satisfying Assumption~\ref{ass: reg of A}. 
The following result seems unhandy, but we prove below that it leads to a large class of coefficient matrices that satisfy the Friedlander--Filonov inequality. 

\begin{proposition}
    \label{prop: frifi for div A nabla}
    Suppose there is a map to the unit sphere $ \xi : \Omega \to S^{d-1} $ such that $ x \mapsto A^{-1/2} (x) \xi (x) $ is a gradient field and $ x \mapsto A^{1/2} (x) \xi (x) $ is divergence-free. 
    Then the eigenvalues of $ L= -\div A \nabla $ satisfy the inequality $ \mu^L_{k+1} \leq \lambda^L_k $ for all $ k \in \N $.
\end{proposition}
\begin{proof}
    Let $ h: \Omega \to \R $ be a function with $ \nabla h = A^{-1/2} \xi $, and fix $ \mu >0 $.
    Then $ w = \e^{ \i \sqrt{ \mu } h } $ is in $ \H^1 ( \Omega ) \setminus \H^1_{ \rho , 0 } ( \Omega ) $ and we verify  
    \begin{equation}
        -\div A \nabla w = ( -\i \sqrt{ \mu } \div A \nabla h 
        + \mu A \nabla h \cdot \nabla h ) \e^{ \i \sqrt{ \mu } h } 
        = ( -\i \sqrt{ \mu } \div A^{1/2} \xi + \mu \abs{ \xi } ) w 
        = \mu w \, , 
    \end{equation}
    as well as 
    \begin{equation}
        \int_\Omega A \nabla w \cdot \nabla \bar{w} 
        = \mu \int_\Omega A A^{-1/2} \xi \cdot A^{-1/2} \xi 
        \abs{w}^2
        = \mu \int_\Omega \abs{w}^2 \, .
    \end{equation}
    So the assertion follows from Proposition~\ref{prop: abstract filonov}.
\end{proof}
If $ \Omega $ is simply connected, the conditions in Proposition~\ref{prop: frifi for div A nabla} read 
\begin{equation}
    \label{eq: div and curl for A}
    \abs{ \xi } \equiv 1 \quad \text{ and }
    \quad \div A^{1/2} \xi = 0 \quad \text{ and } 
    \quad \curl A^{-1/2} \xi = 0 \, . 
\end{equation}
Here, $ \curl $ denotes the differential operator sending a vector field $ F $ to the $ d(d-1)/2 $-dimensional vector with entries $ ( \d_j F_i - \d_i F_j )_{ 1 \leq i<j \leq d } $ (this corresponds to a $ 2 $-form), generalizing the $ 2 $- and $ 3 $-dimensional $ \curl $-operators. 

In general it does not seem easy to determine whether solutions to this problem exist. 
However if there exists a fixed vector $ \xi \in S^{d-1} $ such that both $ A^{1/2} (x) \xi $ and $ A^{-1/2} \xi $ are constant on $ \Omega $, then clearly~\eqref{eq: div and curl for A} is fulfilled. 
This leads to a concrete class of coefficient matrices $ A $ such that $ L = -\div A \nabla $ satisfies the Friedlander--Filonov inequality.
\begin{theorem}
    \label{thm: thm frifi for A}
    Consider the differential operator $ L = -\div A \nabla $, and suppose that $ A $ has a constant eigenvalue $ \lambda > 0 $ with constant eigenvector $ \xi \in \R^d \setminus \{ 0 \} $, that is
    \begin{equation}
        \label{eq: cst eigenpair}
        A(x) \xi = \lambda \xi \quad \text{ for all } x \in \Omega \, . 
    \end{equation}
    Then $ \mu^L_{k+1} \leq \lambda^L_k $ holds for all $ k \in \N $. 
    If $ A $ has two constant eigenvalues with linearly independent constant eigenvectors and if $ A $ is smooth, then $ \mu^L_{k+1} < \lambda^L_k $ for all $ k \in \N $.
\end{theorem}
\begin{proof}
    Let $ \lambda > 0 $ and $ \xi \in \R^d $ with $ \abs{ \xi } = 1 $ and suppose that~\eqref{eq: cst eigenpair} holds. Then also $ x \to A^{1/2} (x) \xi $ and $ x \to A^{-1/2} (x) \xi $ are constant on $ \Omega $. 
    Hence $ A^{-1/2} (x) \xi $ is a gradient field and $ A^{1/2} (x) \xi $ is divergence-free, so Proposition~\ref{prop: frifi for div A nabla} yields the first assertion. 
    
    If $ \eta , \zeta \in \R^{d} $ are two linearly independent such constant eigenvectors and $ E = \spann \{ \eta , \zeta \} \subseteq \R^{d} $, then each $ \xi \in E \cap S^{d-1} $ satisfies~\eqref{eq: div and curl for A}. 
    Moreover the family 
    \begin{equation}
        \label{eq: family inf dim div A grad}
        \sett{ x \mapsto \exp ( A^{-1/2} \xi \cdot x ) : 
        \xi \in E \cap S^{d-1} } 
    \end{equation}
    has infinite dimension. 
    The smoothness of $ A $ guarantees the unique continuation property for $ L $, hence $ L_\text{D} $ and $ L_\text{N}$ do not have a common eigenfunction (see Corollary~\ref{cor: no ef dir neum}). 
    Thus the strict inequality $ \mu^L_{k+1} < \lambda^L_k $ follows like Proposition~\ref{prop: frifi for div A nabla} by Proposition~\ref{prop: abstract filonov}\emph{(iii)}.
\end{proof}
Theorem~\ref{thm: thm frifi for A} should be compared with Theorem~\ref{thm: lw when grad has low dim}. 
In both cases, the coefficients of the elliptic operator $ L $ are constant in certain directions. 
This makes $ L $ sufficiently close to the Laplacian, allowing for certain arguments to carry over. 
\begin{example}
    \label{ex: frifi for A}
    In one space dimension, Theorem~\ref{thm: thm frifi for A} only applies to scalar multiples of $ L = -\ddxtwo $. 
    For $ d \geq 2 $, we get the (non-strict) Friedlander--Filonov inequality for matrices 
    \begin{equation}
        A(x) = 
        \begin{bmatrix}
            1 & 0 \\
            0 & A_1 (x) 
        \end{bmatrix}
        \, , 
    \end{equation}
    and strict inequality for coefficient matrices of the form
        \begin{equation}
        A(x) = 
        \begin{bmatrix}
            a & b & 0 \\
            b & d & 0 \\
            0 & 0 & A_2 (x)
        \end{bmatrix}
        \, .
    \end{equation}
    Here the matrix $ 
    \begin{bmatrix}
        a & b \\
        b & d 
    \end{bmatrix} 
    \in \R^{ 2 \times 2 } $
    is constant and positive definite, the coefficients $ A_1 (x) \in \R^{ (d-1) \times (d-1) } $ and $ A_2 (x) \in \R^{ (d-2) \times (d-2) } $ satisfy Assumption~\ref{ass: reg of A}, and $ A_2 $ is smooth (so that the operator $ L = -\div  A \nabla $ satisfies the unique continuation property).
\end{example}

\appendix

\section{Proof of the Lemma~\ref{lem: ibp convex}}
    \label{sec: proof ibp formula}

In this section we give a proof for Lemma~\ref{lem: ibp convex} with references. 
In~\cite{levine_weinberger} it was proven for smooth domains and functions. We closely follow the lines of~\cite{levine_weinberger} and make sure that their arguments continue to hold in piecewise smooth domains.

Let $ \Omega \subseteq \R^d $ be piecewise smooth and denote again by $ \Sigma \subseteq \d \Omega $ the irregular part of the boundary (see Definition~\ref{def: piecewise smooth domain}). 
The unit normal $ \nu $ is defined on $ \d \Omega \setminus \Sigma $ and admits a smooth extension $ \tilde{ \nu } : C \to S^{d-1} $ for any compact set $ C \subseteq \R^d \setminus \Sigma $ (see~\cite{weingarten_map}, or~\cite[Lemma~6.38]{gilbarg_trudinger} for extending smooth functions defined on the boundary). We define the curvature matrix $ K : \d \Omega \setminus \Sigma \to \R^{ d \times d } $ by 
\begin{equation}
    \label{eq: def of K}
    K_{ij} = ( \d_j - \nu_j \d_\nu ) \tilde{ \nu }_i \, , 
    \quad i,j = 1 , \ldots , n \, . 
\end{equation}
Since the derivative is taken in a direction tangent to $ \d \Omega $, the entry $ K_{ij} $ does not depend on the choice of the extension $ \tilde{ \nu } $, so $ K $ is well-defined. 
Using $ \abs{ \tilde{ \nu } } = 1 $, one easily verifies that $ \tr K = \div \tilde{ \nu } $.
We further define the matrix 
\begin{equation}
    \label{eq: def of B}
    B : \d \Omega \setminus \Sigma \to \R^{ d \times d } \, , 
    \quad B(x) = K(x) + ( \tr K (x) ) \nu(x) \nu(x)^t \, . 
\end{equation}
For $ x \in \d \Omega \setminus \Sigma $, this is a symmetric matrix whose $ d $ eigenvalues are given by the $ (d-1) $ principal curvatures of $ \d \Omega $ at $ x $ and their sum. 
See~\cite[Chapter~7]{elem_diffgeo} for a definition of the Weingarten map (here $ K $) in three dimensions, and~\cite{weingarten_map} for details on representing it via the unit normal. 

We reformulate Lemma~\ref{lem: ibp convex} in the following more general form. Here, we denote by $ DF $ the Jacobi matrix of a vector field $ F $. 
The Betti number $ \beta_n $ of $ \Omega $ are defined as the dimension of the $ n $-the singular homology group $ H_n ( \Omega , \R ) $, see~\cite[Chapter~13]{alg_top} for details. 
\begin{proposition}
    \label{lem: ibp on conv dom}
    Let $ \Omega \subseteq \R^d $ be a piecewise smooth Lipschitz domain, and suppose that $ d=2 $ or that the Betti number $ \beta_{d-2} $ of $ \Omega $ is zero.  
    Let $ F \in \H^1 ( \Omega )^d $ be a curl-free real-valued vector field normal at the boundary.
    Then, for any $ b \in \R^d $, we have the integration-by-parts identity 
    \begin{equation}
        \label{eq: goal in ibp lemma}
        \int_\Omega \abs{ \nabla ( b \cdot F ) }^2 
        = \int_\Omega \div F \, ( b^t D F b ) 
        - \frac{1}{2} \int_{ \d \Omega } \abs{F}^2 \, 
        ( b^t B b ) \, .
    \end{equation}  
\end{proposition}
The condition on the Betti number is probably not necessary.
Indeed, this condition is not needed for smooth vector fields on smooth domains, but only to obtain curl-free approximations of vector fields in $ \H^1 ( \Omega )^d $. If $ d \geq 3 $ and $ \Omega $ is convex or homeomorphic to a ball, then the condition on the Betti number is satisfied. 

To see that Lemma~\ref{lem: ibp convex} follows from Proposition~\ref{lem: ibp on conv dom} it remains to see that $ F = \nabla \varphi $ satisfies the right assumptions when $ \varphi \in \H^2 ( \Omega ) \cap \H^1_0 ( \Omega ) $. 
\begin{lemma}
    \label{lem: grad phi normal}
    Let $ \Omega \subseteq \R^d $ be a piecewise smooth Lipschitz domain and $ \varphi \in \H^2 ( \Omega ) \cap \H^1_0 ( \Omega ) $. 
    Then $ F = \nabla \varphi \in \H^1 ( \Omega ) $ is curl-free and normal at the boundary. 
\end{lemma}
\begin{proof}
    For $ i,j \in \{ 1 , \ldots , d \} $ we have $ \d_i F_j - \d_j F_i = \d_{ij} \varphi - \d_{ji} \varphi = 0 $. 
    
    To prove that $ \nabla \varphi = \d_\nu \varphi \nu $ on $ \d \Omega \setminus \Sigma $, approximate $ \varphi $ in $ \H^2 ( \Omega ) $ by a sequence $ \psi_k \in \Cont_c^\infty ( \R^d ) $. Then $ \psi_k |_{ \d \Omega } \to \varphi |_{ \d \Omega } = 0 $ in $ \H^{3/2} ( \d \Omega ) $.
    Denote by $ E : \H^{3/2} ( \d \Omega ) \to \H^2 ( \Omega ) $ a bounded right-inverse of the trace operator (see~\cite[Chapter~2.5.6]{necas}).
    Then the functions 
    \begin{equation}
        \label{eq: def of approx phi k}
        \varphi_k = \psi_k - E ( \psi_k |_{ \d \Omega } ) 
    \end{equation}
    are smooth, vanish on $ \d \Omega $ and approximate $ \varphi $ in $ \H^2 ( \Omega ) $.
    
    Fix $ k \in \N $. For $ p \in \d \Omega \setminus \Sigma $ and any smooth curve $ \gamma : (-1,1) \to \d \Omega $ with $ \gamma (0) = p $, we have $ \varphi_k \circ \gamma \equiv 0 $. This gives 
    \begin{equation}
        \label{eq: grad phi normal}
        0 = \ddt ( \varphi_k \circ \gamma ) (t) |_{ t=0 } 
        = \nabla \varphi_k (p) \cdot \gamma' (0) \, . 
    \end{equation}
    So $ \nabla \varphi_k (p) $ is perpendicular to every tangential direction at $ p $, hence equal to a scalar multiple of $ \nu (p) $. 
    From $ \nabla \varphi_k |_{ \d \Omega } \to \nabla \varphi |_{ \d \Omega } $ in $ \L^2 ( \d \Omega ) $ we deduce that also $ \nabla \varphi $ is normal on~$ \d \Omega $. 
\end{proof}

For the proof of Proposition~\ref{lem: ibp on conv dom}, we first work in a smooth setting; the computations in the proof of the following lemma appear in~\cite{levine_weinberger}. 
\begin{lemma}
    \label{lem: smooth ibp convex}
    Suppose $ F $ is a smooth curl-free vector field normal at the boundary whose support has positive distance to the irregular part $ \Sigma \subseteq \d \Omega $. 
    Then for $ b \in \R^d $ we have 
    \begin{equation}
        \int_{ \d \Omega } \Big( ( b \cdot F ) \,
        \d_\nu ( b \cdot F ) 
        -  \div F \, ( b \cdot F ) \, ( b \cdot \nu ) \Big) 
        = - \frac{1}{2} \int_{ \d \Omega } \abs{F}^2 \, 
        ( b^t B(x) b ) \, . 
    \end{equation}
\end{lemma}
\begin{proof}
    Fix an extension $ \tilde{ \nu } $ on the support of $ F $, and denote it again by $ \nu $.
    Using $ \d_l F_j = \d_j F_l $ we compute 
    \begin{equation}
        \d_\nu ( b \cdot F ) - \div F \, ( b \cdot \nu ) 
        = \summ_{j,l} b_j ( \nu_l \d_l F_j - \nu_j \d_l F_l )
        = \summ_{j,l} b_j ( \nu_l \d_j - \nu_j \d_l ) F_l \, . 
    \end{equation}
    Since $ F = f \nu $ on $ \d \Omega $ with $ f = F \cdot \nu : \d \Omega \to \R $, writing $ \d^{lj} = ( \nu_l \d_j - \nu_j \d_l ) $ we have
    \begin{multline}
        ( b \cdot F ) ( \d_\nu ( b \cdot F ) 
        - \div F \, ( b \cdot \nu ) )
        = \summ_{i,j,l} b_i b_j f \nu_i 
        \, \d^{lj} ( f \nu_l ) \\
        = \frac{1}{2} \summ_{i,j,l} b_i b_j 
        \d^{lj} ( f^2 \nu_i \nu_l ) 
        - \frac{1}{2} \summ_{i,j} b_i b_j 
        f^2 \left( \summ_l \nu_l \d^{lj} \nu_i 
        - \nu_i \d^{lj} \nu_l \right) \, . 
    \end{multline}
    By the divergence theorem, the boundary integral of the first sum in the last line above vanishes. 
    For the second sum, recalling $ \abs{ \nu } \equiv 1 $ and $ \tr K = \div \nu $ we compute 
    \begin{align}
        \summ_l \nu_l \d^{lj} \nu_i - \nu_i \d^{lj} \nu_l 
        &= \Big( \summ_l \nu_l^2 \Big) \d_j \nu_i 
        - \nu_j \d_\nu \nu_i 
        - \frac{1}{2} \nu_i \d_j \Big( \summ_l \nu_l^2 \Big )
        + \nu_i \nu_j \div \nu \\
        &= K_{ij} + ( \tr K ) \nu_i \nu_j = B_{ij}
    \end{align}
    Integrating over $ \d \Omega $ and recalling $ f^2 = \abs{F}^2 $ we arrive at 
    \begin{equation}
        \int_{ \d \Omega } ( b \cdot F ) 
        ( \d_\nu ( b \cdot F ) 
        - \div F ( b \cdot \nu ) ) 
        = -\frac{1}{2} \summ_{i,j} \int_{ \d \Omega } 
        b_i b_j f^2 B_{ij}  
        = -\frac{1}{2} \int_{ \d \Omega } 
        \abs{F}^2 ( b^t B b ) \, . 
        \qedhere
    \end{equation}
\end{proof}
We now complete the proof Lemma~\ref{lem: ibp on conv dom} via integration by parts, and remove the requirement that $ F $ be smooth by approximation.
\begin{proof}[Proof of Proposition~\ref{lem: ibp on conv dom}]
    Let $ F $ satisfy the assumptions on the theorem. 
    Since $ \Sigma \subseteq \R^d $ has zero capacity (see~\cite[Corollary~5.1.15]{adams_hedberg}), by~\cite[Chapter~8, Theorem~6.3]{edmunds_evans} we find a sequence of smooth vector fields $ G^n \subseteq \Cont_c^\infty ( \R^d )^d $ supported away from $ \Sigma $ such that $ G^n \to F $ in $ \H^1 ( \Omega )^d $. 
    Then $ G^n |_{ \d \Omega } \to F |_{ \d \Omega } $ in $ \H^{1/2} ( \d \Omega )^d $. 
    Since $ F $ is normal at the boundary, the tangential projections satisfy 
    \begin{equation}
        \label{eq: tangent proj goes to zero}
        g^n = G^n |_{ \d \Omega } - ( \nu \cdot G^n |_{ \d \Omega } ) \nu \to 0 
        \quad \text{ in } \H^{1/2} ( \d \Omega )^d 
    \end{equation} 
    (indeed, let $ \Gamma_1 , \ldots , \Gamma_m $ denote the smooth connected components of $ \d \Omega \setminus \Sigma $; as $ G^n $ is supported away from $ \Sigma $ we have $ \norm{ G^n |_{ \d \Omega } }^2_{ \H^{1/2} ( \d \Omega ) } = \summ_i \norm{ G^n |_{ \d \Omega } }^2_{ \H^{1/2} ( \Gamma_i ) } $, and on each smooth component the unit normal is smooth, so the projection is $ \H^{1/2} $-continuous). 
    Denoting by $ E : \H^{1/2} ( \d \Omega ) \to \H^1 ( \Omega ) $ a bounded extension operator (see~\cite[Chapter~2.5.6]{necas}), we obtain a sequence of smooth vector fields $ U^n = G^n - E g^n $ with normal boundary trace, vanishing in a neighborhood of $ \Sigma $ and with $ U^n \to F $ in $ \H^1 ( \Omega )^d $. 
    In particular $ \curl U^n \to \curl F = 0 $ in $ \L^2 ( \Omega ) $. 
    Now consider the boundary value problem 
    \begin{equation}
        \left\{
        \begin{aligned}
            \curl X &= \curl U^n 
            \quad &\text{ in } \Omega \, , \\
            X &= 0 
            \quad &\text{ on } \d \Omega \, .
        \end{aligned}
        \right.
    \end{equation}
    By~\cite[Theorem~3.3.3]{hodge_pdes} there is a smooth solution $ X^n \in \H^1_0 ( \Omega )^d $ with $ \norm{ X^n }_{ \H^1 } \les \norm{ \curl U^n } $.\footnote{Using the notation of~\cite{hodge_pdes}, note that $ \mathbf{t} ( \curl U^n )^\sharp = 0 $ follows from $ U^n $ normal at the boundary (i.e., $ \mathbf{t} U^\sharp = 0 $) and the commutation $ \mathbf{t} \dif = \dif \mathbf{t} $, see~\cite[Propositions~1.2.6 and~3.5.1]{hodge_pdes} for details; finally note that for $ d \geq 3 $, the integrability condition from Theorem~3.3.3 is empty by the condition $ \beta_{d-2} = 0 $, implying $ \mathcal{H}^2_D ( \Omega ) = \{ 0 \} $ (see the remark after Theorem~2.2.2 in~\cite{hodge_pdes}), or for $ d=2 $ by the fact that $ \mathcal{H}^2_D ( \Omega ) = \spann \{ \one \} $ and the fact that $ U^n $ is normal at the boundary.}
    
    Then $ F^n = U^n - X^n $ gives a sequence of smooth curl-free vector fields normal at the boundary approximating $ F $ in $ \H^1 ( \Omega )^d $, and whose supports do not intersect $ \Sigma $. 
    By Lemma~\ref{lem: smooth ibp convex}, the identity~\eqref{eq: goal in ibp lemma} holds true for each $ F^n $.
    Thus integration by parts, $ \curl F^n = 0 $ and Lemma~\ref{lem: smooth ibp convex} give 
    \begin{align}
        \int_\Omega \abs{ \nabla ( b \cdot F^n ) }^2 
        &= -\int_\Omega ( b \cdot \Delta F^n ) \, ( b \cdot F^n ) 
        + \int_{ \d \Omega } ( b \cdot F^n ) 
        \d_\nu ( b \cdot F^n ) \\
        &= - \int_\Omega ( b \cdot \nabla \div F^n ) \, ( b \cdot F^n ) 
        + \int_{ \d \Omega } ( b \cdot F^n ) 
        \d_\nu ( b \cdot F^n ) \\
        &= \int_\Omega \div F^n \, ( b^t D F^n b ) 
        + \int_{ \d \Omega } ( b \cdot F^n ) 
        \d_\nu ( b \cdot F^n ) - \div F^n \, 
        ( b \cdot F^n ) \, ( b \cdot \nu ) \\
        &= \int_\Omega \div F^n \, ( b^t D F^n b ) 
        - \frac{1}{2} \int_{ \d \Omega } 
        \abs{F^n}^2 ( b^t B(x) b ) \, . 
    \end{align}
    Taking limits in $ \H^1 ( \Omega )^d $, the identity continues to hold when $ F^n $ is replaced by $ F $.  
\end{proof}

\subsection*{Competing interests}

The author has no competing interests to declare.

\subsection*{Acknowledgments}

The  author gratefully acknowledges financial support by the grant no. 2022-03342 of the Swedish Research Council (VR).

\end{document}